\documentclass[11pt]{amsart}

\usepackage[T1]{fontenc}
\usepackage{lmodern}
\usepackage{microtype}
\usepackage{amsmath,amssymb}
\usepackage[colorlinks=true,linkcolor=blue,citecolor=blue,urlcolor=blue]{hyperref}

\hypersetup{
  pdftitle={Independence of permutation characters},
  pdfkeywords={finite group, permutation character, cyclic subgroup}
}

\newtheorem{theorem}{Theorem}[section]
\newtheorem{lemma}[theorem]{Lemma}
\newtheorem{corollary}[theorem]{Corollary}
\theoremstyle{definition}

\newtheorem{remark}[theorem]{Remark}

\title[Independence of permutation characters]{On the independence of permutation characters}

\author[M. Brescia]{Mattia Brescia}
\address{Dipartimento di Matematica e Applicazioni ``Renato Caccioppoli'', Universit\`a di Napoli Federico II, Complesso Universitario Monte S. Angelo, Via Cintia, Napoli, Italy}
\email{mattia.brescia@unina.it}

\author[E. Ingrosso]{Ernesto Ingrosso}
\address{Dipartimento di Matematica e Applicazioni ``Renato Caccioppoli'', Universit\`a di Napoli Federico II, Complesso Universitario Monte S. Angelo, Via Cintia, Napoli, Italy}
\email{ernesto.ingrosso2@unina.it}

\author[M. Trombetti]{Marco Trombetti}
\address{Dipartimento di Matematica e Applicazioni ``Renato Caccioppoli'', Universit\`a di Napoli Federico II, Complesso Universitario Monte S. Angelo, Via Cintia, Napoli, Italy}
\email{marco.trombetti@unina.it}


\begin{document}

\begin{abstract}
Let $G$ be a finite group. For every subgroup $H\leq G$, let $\pi_H$ be the permutation character of the action of $G$ on the left cosets of $H$. We prove that the characters $\pi_H$, with $H$ running through representatives of the conjugacy classes of subgroups of $G$, are linearly independent if and only if $G$ is cyclic. In particular, no finite insoluble group has the property asked for in Kourovka Notebook Problem~11.9. The proof uses only the fixed-point formula for a coset action and an elementary triangular-matrix argument.
\end{abstract}

\maketitle

\section{Introduction}

Let $G$ be a finite group and let $H\leq G$. The group $G$ acts on the set of left cosets $G/H$ by left multiplication. The character of this permutation action is denoted by
$$
  \pi_H=\operatorname{Ind}_H^G(1_H).
$$
Thus $\pi_H(g)$ is simply the number of cosets fixed by the element $g$.

Kourovka Notebook Problem~11.9 asks the following.

\medskip

\noindent{\bf Problem}\quad{\it Does there exist a finite insoluble group $G$ for which the permutation characters $\pi_H$, with $H$ running through representatives of the conjugacy classes of subgroups of $G$, are linearly independent over $\mathbb C$?}

\medskip

This problem goes back to the eleventh issue of the
\emph{Kourovka Notebook}, published in 1990. It was attributed to
I.~I.~Pyatetskii-Shapiro and communicated to the Notebook by
Ya.~G.~Berkovich (see \cite[Problem~11.9]{Kourovka}).  The purpose of this note is to give a complete answer. In fact, we
prove the following stronger statement.

\medskip

\noindent{\bf Main Theorem}\quad{\it Let $G$ be a finite group. The permutation characters $\pi_H$, with $H$ running through representatives of the conjugacy classes of subgroups of $G$, are linearly independent over $\mathbb C$ if and only if $G$ is cyclic.}

\medskip

The main observation is simple. The value $\pi_H(g)$ depends only on the cyclic subgroup generated by $g$, up to conjugacy. Hence all permutation characters together can distinguish at most the conjugacy classes of cyclic subgroups. We then show that this upper bound is exact by evaluating the characters associated with cyclic subgroups on generators of those subgroups. The resulting matrix is triangular with non-zero diagonal.

\section{Permutation characters and cyclic subgroups}

Let $G$ be a finite group. For a subgroup $H\leq G$, let $\pi_H$ denote the permutation character of the action of $G$ on $G/H$.

For $g\in G$, a coset $xH$ is fixed by $g$ if and only if
$$
  gxH=xH.
$$
This is equivalent to $x^{-1}gx\in H$. Therefore
\begin{equation}\label{eq:fixed}
  \pi_H(g)
  =\bigl|\{xH\in G/H:x^{-1}gx\in H\}\bigr|.
\end{equation}

We shall use the following notation:
\begin{itemize}
  \item $s(G)$ is the number of conjugacy classes of all subgroups of $G$;
  \item $c(G)$ is the number of conjugacy classes of cyclic subgroups of $G$.
\end{itemize}

\begin{lemma}\label{lem:depends}
For every subgroup $H\leq G$, the value $\pi_H(g)$ depends only on the
conjugacy class of the cyclic subgroup $\langle g\rangle$.
\end{lemma}

\begin{proof}
By \eqref{eq:fixed}, the coset $xH$ is fixed by $g$ if and only if $x^{-1}gx\in H$. Since $H$ is a subgroup, this is equivalent to
$$
  x^{-1}\langle g\rangle x\leq H.
$$
The condition therefore depends only on the conjugacy class of
$\langle g\rangle$.
\end{proof}

\begin{corollary}\label{cor:upper}
The vector space spanned over $\mathbb C$ by all the characters $\pi_H$ has dimension at most $c(G)$.
\end{corollary}

\begin{proof}
By Lemma~\ref{lem:depends}, every permutation character has the same value on two elements whenever the cyclic subgroups generated by those elements are conjugate. There are $c(G)$ possible conjugacy classes of cyclic subgroups. Hence the space of all functions with this property has dimension $c(G)$, and the span of the permutation characters is contained in that space.
\end{proof}

We now prove that the bound in Corollary~\ref{cor:upper} is attained.

\begin{lemma}\label{lem:independent-cyclic}
Choose one representative
$$
  C_1,\ldots,C_{c(G)}
$$
from each conjugacy class of cyclic subgroups of $G$. Then the characters
$$
  \pi_{C_1},\ldots,\pi_{C_{c(G)}}
$$
are linearly independent over $\mathbb C$.
\end{lemma}

\begin{proof}
Order the subgroups so that
$$
  |C_1|\leq |C_2|\leq\cdots\leq |C_{c(G)}|.
$$
For each $j$, choose a generator $g_j$ of $C_j$. Consider the square matrix
$$
  M=\bigl(\pi_{C_i}(g_j)\bigr)_{1\leq i,j\leq c(G)}.
$$

Suppose that $\pi_{C_i}(g_j)\neq0$. By \eqref{eq:fixed}, there is some
$x\in G$ such that $x^{-1}g_jx\in C_i$. Since $g_j$ generates $C_j$, we get
$$
  x^{-1}C_jx\leq C_i.
$$ It follows that $|C_j|\leq |C_i|$.

If $|C_j|=|C_i|$, then the inclusion above is an equality. Thus $C_j$ and $C_i$ are conjugate. Since we chose only one representative from each conjugacy class, this implies $i=j$.

Consequently,
$$
  \pi_{C_i}(g_j)=0 \qquad\text{whenever }j>i.
$$
Thus $M$ is lower triangular.

It remains to compute a diagonal entry. Since $g_i$ generates $C_i$, the condition $x^{-1}g_ix\in C_i$ is equivalent to
$$
  x^{-1}C_ix=C_i,
$$ that is, to $x\in N_G(C_i)$. Two elements of $N_G(C_i)$ determine the same coset of $C_i$ precisely when they differ by an element of $C_i$. Therefore
$$
  \pi_{C_i}(g_i)=|N_G(C_i):C_i|>0.
$$
All diagonal entries of $M$ are non-zero. Hence $M$ is invertible, and the characters $\pi_{C_1},\ldots,\pi_{C_{c(G)}}$ are linearly independent.
\end{proof}

\begin{theorem}\label{thm:rank}
The vector space spanned by all transitive permutation characters of $G$ has dimension $c(G)$.
\end{theorem}

\begin{proof}
Corollary~\ref{cor:upper} gives the upper bound $c(G)$. On the other hand, Lemma~\ref{lem:independent-cyclic} gives $c(G)$ linearly independent permutation characters. Therefore the dimension is exactly $c(G)$.
\end{proof}

\section{Solution of the problem}

\begin{theorem}\label{thm:main}
Let $G$ be a finite group. The characters
$$
  \{\pi_H:H\leq G\text{ up to conjugacy}\}
$$
are linearly independent over $\mathbb C$ if and only if $G$ is cyclic.
\end{theorem}

\begin{proof}
Assume first that the characters are linearly independent. There are $s(G)$ of them, because we take one subgroup from each conjugacy class. By Theorem~\ref{thm:rank}, their span has dimension $c(G)$. Linear independence therefore gives
$$
  s(G)\leq c(G).
$$
Since every cyclic subgroup is a subgroup, we always have
$$
  c(G)\leq s(G).
$$
Hence $s(G)=c(G)$.

The conjugacy classes of cyclic subgroups form a subset of the conjugacy classes of all subgroups. Since the two sets have the same finite cardinality, every subgroup of $G$ is conjugate to a cyclic subgroup, and is therefore cyclic. In particular, the subgroup $G$ itself is cyclic.

Conversely, assume that $G$ is cyclic. Every subgroup of a cyclic group is cyclic. Hence the full family of characters is the family considered in Lemma~\ref{lem:independent-cyclic}, and is linearly independent.
\end{proof}

\begin{remark}
The proof does not use solubility. The obstruction is simply that permutation characters can distinguish no more than the conjugacy classes of cyclic subgroups, whereas the family in the problem is indexed by all conjugacy classes of subgroups.
\end{remark}

\bigskip\bigskip

\noindent{\bf Acknowledgements}\\ The authors are members of the non-profit association ``AGTA---Advances in Group Theory and Applications'' (www.advgrouptheory.com) and are supported by GNSAGA (INdAM).  Moreover, Brescia and Trombetti were supported by the FRA project FORMALG of the University of Naples Federico II (CUP E65F22000060001).

\end{document}